\documentclass[12pt,letterpaper,reqno]{amsart}

\usepackage{amssymb,amsmath}
\usepackage{amsthm}
\usepackage{color,graphicx}
\usepackage{eucal}

\newtheorem{theorem}{Theorem}[section]
\newtheorem{proposition}[theorem]{Proposition}
\newtheorem{remark}[theorem]{Remark}
\newtheorem{lemma}[theorem]{Lemma}
\newtheorem{corollary}[theorem]{Corollary}
\newtheorem{definition}[theorem]{Definition}

\newcommand{\R}{\mathbb R}

\newcommand{\la}{{\lambda}}

\newcommand{\supp}{\textmd{supp}}

\newcommand{\al}{\alpha}

\newcommand{\cR}{\mathbb R}

\newcommand{\de}{\delta}
\newcommand{\ep}{\varepsilon}

\newcommand{\ds}{\displaystyle}

\numberwithin{equation}{section}

\begin{document}

\title[Instablity of solitons in gKdV]
{Instability of solitons - revisited, I:\\
the critical generalized KdV equation}

\author[L. G. Farah]{Luiz Gustavo Farah}
\address{Department of Mathematics\\UFMG\\Brazil}
\curraddr{}
\email{lgfarah@gmail.com}
\thanks{}

\author[J. Holmer]{Justin Holmer}
\address{Department of Mathematics\\Brown University\\ USA}
\email{holmer@math.brown.edu}
\thanks{}

\author[S. Roudenko]{Svetlana Roudenko}
\address{Department of Mathematics\\The George Washington University\\USA}
\curraddr{}
\email{roudenko@gwu.edu}
\thanks{} 




\begin{abstract}
We revisit the phenomenon of instability of solitons in the generalized Korteweg-de Vries equation, $u_t + \partial_x(u_{xx} + u^p) = 0$. It is known that solitons are unstable for nonlinearities $p \geq 5$, with the critical power $p=5$ being the most challenging case to handle. This was done by Martel-Merle in  \cite{MM-KdV-instability}, where the authors crucially relied on the pointwise decay estimates of the linear KdV flow. In this paper, we show simplified approaches to obtain the instability of solitons via truncation and monotonicity, which can be also useful for other KdV-type equations.
\end{abstract}



\maketitle


\tableofcontents

\section{Introduction}

In this paper we consider the $L^2$-critical generalized KdV equation:
\begin{equation}
\label{gKdV}
\text{(gKdV)} \qquad \qquad
u_t + u_{xxx} + (u^5)_x = 0,  \qquad x \in \R, \qquad t \in \R.
\end{equation}
During their lifespan, the solutions $u(t, x)$ to \eqref{gKdV} conserve the mass and energy:
\begin{equation*}
M[u(t)]=\int_{\cR} u^2(t)\, dx = M[u(0)]
\end{equation*}
and
\begin{equation*}
E[u(t)]=\dfrac{1}{2}\int_{\cR}|\nabla u(t)|^2\;dx - \dfrac{1}{6}\int_{\cR} u^{6}(t)\;dx = E[u(0)].
\end{equation*}
Also, for solutions
$u(t,x)$ decaying at infinity on $\cR$ the following invariance holds
\begin{equation*}
\int_{\cR} u(t, x) \, dx = \int_{\cR} u(0, x) \, dx,
\end{equation*}
which is obtained by integrating the original equation on $\cR$.

For the existence of solutions, one typically considers the Cauchy problem
\begin{equation}
\label{gKdV-p}
\left\{
\begin{array}{l}
u_t + u_{xxx} + (u^p)_x = 0,  \qquad (x, t) \in \R \times \R \\
u(t,x) = u_0 \in H^s(\R),
\end{array}
\right.
\end{equation}
with $s \geq \R$ and $p$ being an integer (although it is possible to consider continuous power $p$ with the nonlinear term replacement $\partial_x(|u|^{p-1}u)$, however, the equation in the odd power cases would differ slightly from $\partial_x(u^p)$, nevertheless, the well-posedness theory would work the same.
For this paper we are only interested in the initial data in $H^1(\R)$ space, and the local well-posedness sufficient for this case is available from the classical work of Kenig-Ponce-Vega \cite{KPV-93} (in fact, the $L^2$ theory for the equation \eqref{gKdV}, and sharp $H^s$ theory for \eqref{gKdV-p}).

The gKdV equation has a family of travelling wave solutions (often referred as solitary waves or even solitons),
which are of the form
\begin{equation}\label{gKdV-soliton}
u(t,x) = Q_c(x-ct), c>0
\end{equation}
with $Q_c(x) \to 0$ as $|x| \to + \infty$ and $Q_c$ is the dilation of the ground state
$$
Q_c({x}) = c^{1/{4}} Q (c^{1/2} {x}),
$$
where $Q$ is a radial positive solution in $H^1(\cR)$ of the well-known nonlinear elliptic equation
$$
-Q + Q_{xx} + Q^p = 0.
$$
Note that $Q \in C^{\infty}(\R)$, $\partial_r Q(r) <0$ for any $r = |x|>0$ and it is exponentially decaying at infinity 
\begin{equation}\label{prop-Q}
|\partial^\al Q({x})| \leq c(\al)\, e^{- |{x}|} \quad \mbox{for any}\quad {x} \in \cR.
\end{equation}
The questions about stability of travelling waves \eqref{gKdV-soliton} have been one of the key features in the gKdV theory, and have attracted a lot of attention in the last twenty years.

The purpose of this note is to review approaches available to study stability and instability questions in the generalized KdV equation, specifically, in the critical case. The criticality notion comes from the scaling symmetry of the equation \eqref{gKdV-p},
which states that an appropriately rescaled version of the original solution is also a solution of the equation.
For the equation \eqref{gKdV} it is
$$
u_\lambda(t, x)=\lambda^{\frac{1}{2}} u(\lambda^3t, \lambda x).
$$
This symmetry makes the $L^2$-norm invariant, i.e.,
\begin{equation*}
\|u_\lambda(0,\cdot) \|_{L^2}= \|u_0\|_{L^2},
\end{equation*}
and thus, the reference of \eqref{gKdV} as the $L^2$-critical equation. (There are also other symmetries such as translation and dilation.)

The original breakthrough for the critical gKdV equation in obtaining the instability of travelling waves (which later led to the existence of blow-up solutions) was done by Martel and Merle in \cite{M-KdV}, which heavily rely on \cite{MM-KdV1}, for the first blow-up result refer to the paper by Frank Merle \cite{MM-KdV1}, also see \cite{MM-KdV3}-\cite{MM-KdV4}. In other cases ($p \neq 5$) the stability of solitons (as well as asymptotic stability) is known for $p<5$, for example, see \cite{MM-asym}; for classical orbital stability results refer to \cite{GSS} and \cite{BSS}, where it was also shown the instability of solitons for $p>5$. The supercritical case $p>5$ was revisited by Combet in \cite{Combet}, where among other things he gave a nice argument of instability via the so-called monotonicity properties, see \cite[\S 2.3]{Combet}. This note was partially motivated by his argument.

The motivation of this paper is two-fold: one is to revisit the instability in the critical case, as it is the most challenging, and show simplified approaches to obtain it (in the original proof of \cite{MM-KdV-instability} the authors crucially relied on the pointwise decay estimates of the linear shifted KdV flow, and then made a double in time application of them  to the nonlinear problem); here, we show two proofs: one via the truncation and monotonicity, and the second one is only relying on monotonicity. Another purpose of this review is to set the stage for other generalizations of the gKdV equation, for example, for BO, BBM, KP-type equations, including the higher dimensional generalizations such as Zakharov-Kuznetsov (ZK) equation, the supercritical case of which we investigate in part II of this project \cite{FHR2}.

We now give the precise concept of stability and instability of solitons used in this work. For $\alpha>0$, the neighborhood (or ``tube") of radius $\alpha$ around $Q$ (modulo translations) is given by
\begin{equation}\label{tube}
U_{\alpha}=\left\{u\in H^1(\mathbb{R}): \inf_{{y}\in \mathbb{R}}\|u(\cdot)-Q(\cdot+{y})\|_{H^1}\leq \alpha \right\}.
\end{equation}
\begin{definition}[Stability of $Q$]
We say that $Q$ is stable if for all $\alpha>0$, there exists $\delta>0$ such that if $u_0\in U_{\delta}$, then the corresponding solution $u(t)$ is defined for all $t\geq 0$ and $u(t)\in U_{\alpha}$ for all $t\geq 0$.
\end{definition}

\begin{definition}[Instability of $Q$]
We say that $Q$ is unstable if $Q$ is not stable, in other words, there exists $\alpha>0$ such that for all $\delta>0$ such that if $u_0\in U_{\delta}$, then there exists $t_0=t_0(u_0)$ such that $u(t_0)\notin U_{\alpha}$.
\end{definition}

The main result of this paper, which we revisit and show different ways to prove, reads as follows.
\begin{theorem}[$H^1$-instability of $Q$ for the critical gKdV]\label{Theo-Inst}
There exists $0<\alpha_0, b_0<1$ such that if $u_0=Q+\ep_0$, with $\ep_0\in H^1(\R)$ satisfying
$$
\|\ep_0\|^2_{H^1}\leq b_0\int \ep_0 Q, \quad \ep_0 \perp \{Q_{y}, Q^3\}
$$
and
\begin{equation}\label{E:ep0}
|\ep_0(y)|\leq {c} \, e^{-{\delta}|y|} \textrm{ for some } {c}>0 \textrm{ and } {\delta}>0,
\end{equation}
then there exists $t_0=t_0(u_0)$ such that $u(t_0)\notin U_{\alpha_0}$, or explicitly,
$$
\inf_{{y}\in \R}\|u(t_0, \cdot)-Q(\cdot-{y})\|\geq \alpha_0.
$$
\end{theorem}
\begin{remark}
We note that in the above version of the instability statement, we use the initial data \eqref{E:ep0} with $\ep_0$ decaying exponentially. In the original proof of \cite{MM-KdV-instability} the authors showed the instability on a larger class data, i.e., with $\ep_0$ decaying polynomially of a certain degree specified. For the purpose of showing the instability phenomenon, it is not important how large the set of initial data is, in fact, as it was pointed out in \cite{Combet}, it is sufficient to exhibit one example (for example, a sequence, converging to the solution with the needed properties). On our set of initial data, it is easier to show the $L^2$ exponential decay on the right of the soliton, which gives the monotonicity property, the crucial but simple ingredient in our proof. To be able to show the instability on the polynomially decaying initial data, Martel-Merle had to use the pointwise decay estimates in \cite{MM-KdV-instability}.
\end{remark}

The paper is organized as follows. In Section \ref{S-2} we review the properties of the ground state $Q$ together with the linearized equation around it as well as the modulation theory. In Section \ref{S-3} we discuss the concept of monotonicity, and as a consequence, the $L^2$ exponential decay on the right of the soliton. In Section \ref{S-4} we revisit the virial-type estimates. The next Section \ref{S-5} contains the proof of the instability via truncation and monotonicity, and in the last Section \ref{S-6} we give an alternative proof relying only on monotonicity (and without any truncation).
\medskip

{\bf Acknowledgements.} Most of this work was done when the first author was visiting GWU in 2016-17 under the support of the Brazilian National Council for Scientific and Technological Development (CNPq/Brazil), for which all authors are very grateful as it boosted the energy into the research project.
S.R. would like to thank MSRI for the excellent working
conditions during the semester program ``New Challenges in PDE : Deterministic Dynamics and Randomness
in High and Infinite Dimensional Systems" in the Fall 2015, in particular, she would like to thank Ivan Martel for the discussions on the topic as well as all organizers of the program.
L.G.F. was partially supported by CNPq and FAPEMIG/Brazil.
J.H. was partially supported by the NSF grant DMS-1500106.
S.R. was partially supported by the NSF CAREER grant DMS-1151618.

\section{Preliminaries on $Q$ and the linearization around it} \label{S-2}

We start with considering the canonical parametrization of the solution $u(t,{x})$:
\begin{equation*}
v(t,y) = \la(t)^{1/2} \, u(t, \la(t)y +x(t)),
\end{equation*}
and since we will be studying solutions close to $Q$, we define their difference $\ep = v-Q$ by
\begin{equation*}
\ep(t,{y}) = v(t,{y}) - Q({y}).
\end{equation*}
We rescale the time $t \mapsto s$ by $\ds \frac{ds}{dt} = \frac1{\la^3}$, so $\ep = \ep(s,y)$.

\subsection{The linearized equation around $Q$ and its properties}
We have the following equation for $\ep$.
\begin{lemma}\label{eq-ep}
For all $s \geq 0$, we have
\begin{equation}\label{ep1-order1}
\ep_s = 
(L \ep)_{y}  + \frac{\la_s}{\la} \Lambda Q + \left( \frac{x_s}{\la} -1 \right) Q_{y}
+ \frac{\la_s}{\la} \Lambda \ep + \left( \frac{x_s}{\la} -1 \right) \ep_{y}  -  R(\ep)_{y},
\end{equation}
where the generator $\Lambda$ of scaling symmetry is defined by
\begin{equation}\label{Eq:LambdaQ}
\Lambda f = \frac{1}{2} f + {y} \cdot f_y,
\end{equation}
and $L$ is the linearized operator around $Q$
\begin{equation}\label{Eq:L}
L \ep = - \ep_{xx} + \ep - 5 Q^4 \ep,
\end{equation}
and the higher order in $\ep$ remainder $R(\ep)$ is given by
\begin{equation}\label{R(ep)}
R(\ep)= 10Q^3 \ep^2+10Q^2\ep^3+5Q\ep^4+\ep^5.
\end{equation}
\end{lemma}
\begin{proof}
It is a straightforward computation, also refer to Martel-Merle \cite[Lemma 1]{MM-KdV-instability}.
\end{proof}

\subsection{Spectral properties of $L$}
It is important to recall the properties of the operator $L = - \partial_{xx} + 1 -5 Q^4 $ (see Kwong \cite{K89} for all dimensions, Weinstein \cite{W85} for dimension 1 and 3, also Maris \cite{M02} and \cite{CGNT}).
\begin{theorem}
The following holds for an operator $L$ defined in \eqref{Eq:L}:
\begin{enumerate}
\item
$L$ is a self-adjoint operator and
$\sigma_{ess}(L) = [ \la_{ess}, +\infty ) \quad \mbox{for~~ some~~} \la_{ess} > 0$
\item
$\ker L = \mbox{span} \{Q_{y} \}$
\item
$L$ has a unique single negative eigenvalue $-8$ associated to a
positive radially symmetric eigenfunction $Q^3$. Moreover, there exists $\de_0 > 0$ such that
$|Q^3(x)| \lesssim e^{- \de_0 |x|}$
for all $x \in \cR$ (which is obvious in the light of \eqref{prop-Q}).
\end{enumerate}
\end{theorem}
In general, the operator $L$ is not positive-definite, however, the following lemma shows that if certain directions are removed from consideration, then it becomes positive.
\begin{lemma}\label{Lemma-ort3}
For any $f \in H^1(\R)$ such that
\begin{equation}\label{Ort-Cond2}
(f, Q^3) = (f, Q_{x}) =0,
\end{equation}
one has
$$
(Lf,f) \geq \, \|f\|_2^2.
$$
\end{lemma}
\begin{proof}
See \cite[Lemma 2]{MM-KdV-instability}.
\end{proof}
We also observe that
\begin{equation*}
L(\Lambda Q) = -2Q,
\end{equation*}
where $\Lambda$ is defined in \eqref{Eq:LambdaQ}, furthermore, in this $L^2$-critical case, $(\Lambda Q, Q) = 0$ (this is one of the reasons that the argument of Combet \cite{Combet} to show the instability in the supercritical gKdV does not work in the critical case).

\subsection{Conservation laws for $\ep$}
Our next item is to derive the mass and energy conservation for $\ep$. Denoting
\begin{equation}\label{M_0}
M_0 = 2 \int_{\cR} Q({y}) \ep(0,{y}) \, d{y} + \int_{\cR} \ep^2(0,{y}) \, d{y},
\end{equation}
in other words, $M[Q+\ep(0)] = M_0 + M[Q]$, we have that the following
result holds.
\begin{lemma}
For any $s \geq 0$, the mass and energy of $\ep$ are conserved as follows
\begin{equation}
\label{ep-M+E}
M[\ep(s)] = M_0 \quad \mbox{and} \quad E[Q+\ep(s)] = \la^2(s) \, E[u_0].
\end{equation}
Moreover, the energy linearization is
\begin{align}\label{E-lin}
E[Q+\ep]  +  \left( \int Q \ep + \frac12 \int \ep^2 \right) =& \frac12 (L\ep,\ep) \nonumber \\
&- \frac16 \left( 20 \int Q^3\ep^3 +  15\int Q^2\ep^4+6\int Q\ep^5+\int\ep^6\right),
\end{align}
and if $\|\ep \|_{H^1} \leq 1$, then there exists $c_0 >0$ such that
\begin{equation*}
\left| E[Q+\ep]  +  \left( \int Q \ep + \frac12 \int \ep^2 \right) - \frac12 (L\ep,\ep) \right| \leq c_0 \, \| \ep \|_{H^1} \| \ep \|^2_{L^2}.
\end{equation*}
\end{lemma}
\begin{proof}
See \cite[Lemma 3]{MM-KdV-instability}.
\end{proof}

\subsection{Modulation Theory}
Now that we know how to make the operator $L$ positive (from Lemma \ref{Lemma-ort3},
to enforce the orthogonality conditions \eqref{Ort-Cond2}, we use the modulation (which is proved via the implicit function theorem).
Given $\lambda_1, x_1 \in \R$ and $u\in H^1(\R)$, we define
\begin{equation}\label{ep-def2}
\ep_{\lambda_1,{x}_1}(y) = \lambda_1^{1/2} \, u( \lambda_1y +x_1) - Q(y),
\end{equation}
Observe that if $u$ is small in the sense of definition \eqref{tube}, then it is possible to choose parameters $\lambda_1, x_1\in \mathbb{R}$ such that $\ep_{\lambda_1,{x}_1} \perp Q^3$ and $\ep_{\lambda_1,{x}_1} \perp Q_{y}$.

\begin{proposition}[Modulation Theory]\label{ModThI}
There exists $\overline{\alpha},  \overline{\lambda} >0$ and a unique $C^1$ map
$$
(\lambda_1, {x}_1): U_{\overline{\alpha}}\rightarrow (1-\overline{\lambda}, 1+\overline{\lambda})\times \mathbb{R}
$$
such that if $u\in U_{\overline{\alpha}}$ and $\ep_{\lambda_1,{x}_1}$ is given by \eqref{ep-def2},
then
\begin{equation*}
\ep_{\lambda_1, {x}_1} \perp Q^3 \quad \textrm{and} \quad \ep_{\lambda_1, {x}_1} \perp Q_{y}.
\end{equation*}
Moreover, there exists a constant $C_1>0$, such that if $u\in U_{\alpha}$, with $0<\alpha<\overline{\alpha}$, then
\begin{equation*}
\|\ep_{\lambda_1, {x}_1}\|_{H^1}\leq C_1\alpha \quad \textrm{and} \quad |\lambda_1-1|\leq C_1\alpha.
\end{equation*}
\end{proposition}

\begin{proof}
See \cite[Proposition 1]{MM-KdV-instability}.
\end{proof}

Now, assume that $u(t)\in U_{\overline{\alpha}}$, for all $t\geq 0$. We define the functions $\lambda$ and $x$ as follows
\begin{definition}\label{eps}
For all $t\geq 0$, let $\lambda(t)$ and $x(t)$ be such that $\ep_{\lambda(t), x(t)}$ defined according to equation \eqref{ep-def2} satisfy
\begin{equation*}
\ep_{\lambda(t), x(t)} \perp Q^3 \quad \textrm{and} \quad \ep_{\lambda(t), x(t)} \perp Q_{y}.
\end{equation*}
In this case we also define
\begin{equation*}
\ep(t)=\ep_{\lambda(t), x(t)}=\lambda^{1/2}(t) \, u( \lambda(t)y +x(t)) - Q(y).
\end{equation*}
\end{definition}

\subsection{Estimates on Parameters}
To get a more precise control of the parameters $x(t)$ and $\lambda(t)$, we again rescale the time $t \mapsto s$ by $\frac{ds}{dt} = \frac1{\la^3}$. Indeed, the following proposition provides us with the equations and estimates for $\dfrac{\lambda_s}{\lambda}$ and $\left(\dfrac{x_s}{\lambda}-1\right)$.

\begin{lemma}[Modulation parameters]\label{Lemma-param}
There exists $0<\alpha_1<\overline{\alpha}$ such that if for all $t\geq 0$, $u(t)\in U_{\alpha_1}$, then $\lambda$ and ${x}$ are $C^1$ functions of $s$ and they satisfy the following equations
\begin{eqnarray*}
\frac{\lambda_s}{\lambda}\left(\frac14 \int Q^4-\int y(Q^3)_y\ep\right) - \left(\frac{x_s}{\lambda}-1\right)\int (Q^3)_{y}\ep \nonumber \\
=\int L((Q^3)_{y})\ep- \int (Q^3)_{y}R(\ep),
\end{eqnarray*}
and
\begin{eqnarray*}
-\frac{\lambda_s}{\lambda}\int yQ_{yy}\ep- \left(\frac{x_s}{\lambda}-1\right)\left(\frac12\int Q^2- \int Q_{yy}\ep\right) \nonumber \\
=20\int Q^3 Q^2_y\ep - \int Q_{yy}R(\ep).
\end{eqnarray*}

Moreover, there exists a universal constant $C_2>0$ such that\\
if $\|\ep(s)\|_2\leq \alpha$ for all $s\geq 0$, where $\alpha<\alpha_1$, then
\begin{equation}\label{ControlParam}
\left|\frac{\lambda_s}{\lambda}\right|+\left|\frac{x_s}{\lambda}-1\right|\leq C_2\|\ep(s)\|_2.
\end{equation}
\end{lemma}
\begin{proof}
See \cite[Lemmas 4 and 12]{MM-KdV-instability}.
\end{proof}
We are now ready to discuss the approach of monotonicity.

\section{Monotonicity}\label{S-3}
As it was remarked in \cite{M-KdV}, the solution $u(x,t)$ has the mass around its center $x(t)$ versus the mass to the right of the soliton being a time decreasing function. The solution to the right of the soliton is small in $L^2$ sense. To demonstrate that, we first introduce the $L^2$ localization functional $I$, see \eqref{Eq:I}, and then show the exponential decay of $L^2$ norm to the right of the soliton.
For $M\geq 4$, denote
$$
\psi(x)=\frac{2}{\pi}\arctan{(e^{\frac{x}{M}})}.
$$
It is easy to check the following properties:
\begin{enumerate}
\item
$\displaystyle \psi(0)=\frac{1}{2}$,
\item
$\displaystyle \lim_{x\rightarrow -\infty} \psi(x)=0$ and $\displaystyle \lim_{x\rightarrow +\infty} \psi(x)=1$,
\item
$\displaystyle \psi^{'}(x)=\left(\pi M \cosh\left(\frac{x}{M}\right)\right)^{-1}$,
\item $\displaystyle \left| \psi^{\prime\prime\prime}(x)\right|\leq \frac{1}{M^2}\psi^{\prime}(x)\leq \frac{1}{16}\psi^{\prime}(x)$.
\end{enumerate}

Let $x(t)\in C^1(\R, \R)$, and for $x_0, t_0>0$ and $t\in [0,t_0]$ define the functional
\begin{equation}\label{Eq:I}
I_{x_0,t_0}(t)=\int u^2(t, x) \,\psi\left(x-x(t_0)+\frac{1}{2}(t_0-t)-x_0\right)dx,
\end{equation}
where $u\in C(\R, H^1(\R))$ is a solution of the critical gKdV equation \eqref{gKdV}, satisfying
\begin{equation}\label{u-Q}
\|u(\cdot+{x}(t))-Q\|_{H^1}\leq \alpha \quad \mbox{for ~ some} \quad \alpha>0.
\end{equation}

While the next property is basically known, we provide it here for completeness.
\begin{lemma}[Almost Monotonicity]\label{AM}
Let $M\geq 4$ be fixed and assume that $x(t)$ is an increasing function satisfying $x(t_0)-x(t)\geq \frac{3}{4}(t_0-t)$ for every $t_0, t\geq 0$ with $t\in [0,t_0]$. Then there exist $\alpha_0>0$ and $\theta=\theta(M)>0$ such that if $u\in C(\R, H^1(\R))$ verify \eqref{u-Q} with $\alpha<\alpha_0$, then for all $x_0>0$ and $t_0, t\geq 0$ with $t\in [0,t_0]$, we have
$$
I_{x_0,t_0}(t_0)-I_{x_0,t_0}(t)\leq \theta e^{-\frac{x_0}{M}}.
$$
\end{lemma}
\begin{proof}
Using the equation and the fact that $ \left| \psi^{\prime\prime\prime}(x)\right|\leq \frac{1}{M^2}\psi^{\prime}(x)\leq \frac{1}{16}\psi^{\prime}(x)$, we obtain the following bound on the derivative of the functional $I_{x_0,t_0}(t)$:
\begin{align}\label{Ix0t0}
\nonumber
\frac{d}{dt} I_{x_0,t_0}(t) \leq &2 \int uu_t \psi - \frac{1}{2} \int u^2\psi^{\prime}\\
\nonumber
\leq & -\int \left(3u_x^2-\frac{5}{3}u^6\right)\psi^{\prime}+\int u^2\psi^{\prime\prime\prime}-\frac{1}{2}\int u^2\psi^{\prime}\\
\leq & -\int \left(3u_x^2+\frac{1}{4}u^2\right)\psi^{\prime}+\frac{5}{3}\int u^6\psi^{\prime}.
\end{align}
Estimating the last term on the right hand side of the last inequality, we view it in terms of closeness to $Q$ and write
\begin{align}\label{RHS}
\int u^6\psi^{\prime}=\int Q(\cdot-x(t))u^5\psi^{\prime}+\int (u-Q(\cdot-x(t)))u^5\psi^{\prime}.
\end{align}
Application of the Sobolev Embedding $H^1(\R)\hookrightarrow L^{\infty}(\R)$ in the second term above yields
\begin{align}\label{RHS1}
\nonumber
\int (u-Q(\cdot-x(t)))u^5\psi^{\prime}\leq& \|u-Q(\cdot-x(t))\|_{\infty}\|u\|^3_{\infty}\int u^2\psi^{\prime}\\
\leq & c \alpha\|Q\|^3_{H^1} \int u^2\psi^{\prime}.
\end{align}
For the first term on the right hand side of \eqref{RHS}, we divide the integration into two regions $|x-x(t)|\geq R_0$ and $|x-x(t)|< R_0$, where $R_0$ is a positive number to be chosen later. Therefore, since $Q(x)\leq ce^{-|x|}$, we get
\begin{align*}
\nonumber
\int_{|x-x(t)|\geq R_0} Q(\cdot-x(t))u^5\psi^{\prime}
\leq & c \,e^{-R_0}\|u\|^3_{\infty}\int u^2\psi^{\prime}\\
\leq & c \, e^{-R_0}\|Q\|^3_{H^1} \int u^2\psi^{\prime}.
\end{align*}
When $|x-x(t)|< R_0$, we estimate
\begin{align*}
\nonumber
\left|x-x(t_0)+\frac{1}{2}(t_0-t)-x_0\right| \geq& (x(t_0)-x(t)+x_0)-\frac{1}{2}(t_0-t)-|x-x(t)|\\
\geq & \frac{1}{4}(t_0-t)+x_0-R_0,
\end{align*}
where in the first inequality we used that $x(t)$ is increasing, $t_0\geq t$ and $x_0>0$, to compute the modulus of the first term and in the second line we used the assumption $x(t_0)-x(t)\geq \frac{3}{4}(t_0-t)$.

Now, since $\psi^{\prime}(z)\leq c \,e^{-\frac{|z|}{M}}$, we deduce that
\begin{align}\label{RHS3}
\nonumber
\int_{|x-x(t)|\leq R_0} Q(\cdot-x(t))u^5\psi^{\prime}
\leq & \|u\|^5_{\infty}\|Q\|_1e^{\frac{R_0}{M}}e^{-\frac{\left(\frac{1}{4}(t_0-t)+x_0\right)}{M}}\\
\leq & c \,\|Q\|^5_{H^1}\|Q\|_1e^{\frac{R_0}{M}}e^{-\frac{\left(\frac{1}{4}(t_0-t)+x_0\right)}{M}}.
\end{align}
Therefore, choosing $\alpha$ such that $c \, \alpha\|Q\|^3_{H^1}<\frac{3}{5} \cdot \frac{1}{4}$ and $R_0$ such that $c \, \|Q\|^3_{H^1}e^{-R_0}<\frac{3}{5} \cdot \frac{1}{4}$, collecting \eqref{RHS1}-\eqref{RHS3}, we obtain
$$
\frac{5}{3}\int u^6\psi^{\prime}\leq \frac{1}{8}\int u^2\psi^{\prime}+c \, \|Q\|^5_{H^1}\|Q\|_1e^{\frac{R_0}{M}}\, e^{-\frac{\left(\frac{1}{4}(t_0-t)+x_0\right)}{M}}.
$$
Inserting the previous estimate into \eqref{Ix0t0}, we get that there exists $C>0$ such that
\begin{align*}
\nonumber
\frac{d}{dt} I_{x_0,t_0}(t) \leq &-\int \left(3u_x^2+\frac{1}{8}u^2\right)\psi^{\prime}+c\,e^{-\frac{x_0}{M}}\cdot e^{-\frac{1}{4M}(t_0-t)}\\
\nonumber
\leq & c \, e^{-\frac{x_0}{M}}\cdot e^{-\frac{1}{4M}(t_0-t)}.
\end{align*}
Finally, integrating on $[t,t_0]$, we obtain the desired inequality for some $\theta=\theta(M)>0$.
\end{proof}

We now show the exponential decay of the $L^2$ norm to the right of the soliton.
\begin{lemma}\label{AM2}
Let $x(t)$ satisfy the assumptions of Lemma \ref{AM} and assume that $x(t)\geq \frac{1}{2}t$ for all $t\geq 0$. Let $u\in C(\R, H^1(\R))$ be a solution of the gKdV equation \eqref{gKdV} satisfying \eqref{u-Q} with $\alpha<\alpha_0$ (with $\alpha_0$ given in Lemma \ref{AM}) and with initial data $u_0$ verifying $|u_0(x)|\leq c\, e^{-\delta|{x}|}$ for some $c>0$ and $\delta>0$. Fix $M\geq \max\{4, \frac{1}{\delta}\}$. Then there exists $C=C(M,\delta)>0$ such that for all $t\geq 0$ and $x_0>0$
\begin{equation}\label{Eq:L2-right}
\int_{x>x_0}u^2(t, x+x(t))\, dx\leq C\, e^{-\frac{x_0}{M}}.
\end{equation}
\end{lemma}
\begin{proof}
From Lemma \ref{AM} with $t=0$ and replacing $t_0$ by $t$, we deduce that for all $t\geq 0$
$$
I_{x_0,t}(t)-I_{x_0,t}(0)\leq \theta \, e^{-\frac{x_0}{M}}.
$$
This is equivalent to
$$
\int u^2(t, x)\psi(x-x(t)-x_0) \, dx \leq \int u_0^2(x)\psi(x-x(t)+\frac{1}{2}t-x_0) \, dx + \theta \, e^{-\frac{x_0}{M}}.
$$
On the other hand,
$$
\int u^2(t, x)\psi(x-x(t)-x_0)\,  dx = \int u^2(t, x+x(t))\psi(x-x_0) \, dx \geq \frac{1}{2}\int_{x>x_0} u^2(t, x+x(t))\, dx,
$$
where in the last inequality we have used the fact that $\psi$ is increasing and $\psi(0)=1/2$.

Now, since $-x(t)+\frac{1}{2}t\leq 0$ and $\psi$ is increasing, we get
$$
\int u_0^2(x)\psi(x-x(t)+\frac{1}{2}t-x_0) \, dx \leq \int u_0^2(x)\psi(x-x_0) \, dx.
$$
The assumptions $|u_0(x)|\leq c\, e^{-\delta|{x}|}$ and $\psi(x)\leq c \, e^{\frac{x}{M}}$ yield
$$
\int u_0^2(x)\psi(x-x_0) \, dx \leq c \, \int e^{-2\delta|{x}|} e^{\frac{x-x_0}{M}} \, dx
\leq c \, e^{-\frac{x_0}{M}}\int e^{-\left(2\delta|x|-\frac{x}{M}\right)} \, dx
\leq \bar{C}(M,\delta) \, e^{-\frac{x_0}{M}},
$$
where in the last inequality we used the fact that
$$
2\delta -\frac{1}{M}\geq \delta \Longleftrightarrow M\geq \frac{1}{\delta}.
$$
Collecting the above estimates, we obtain the inequality \eqref{Eq:L2-right}.
\end{proof}

\section{Virial-type estimates} \label{S-4}

In this section, we define a quantity depending on the $\ep$ variable that plays an important role in our instability proof. Indeed, let $\varphi \in C_0^{\infty}(\R)$ be a decreasing function with
\begin{equation*}
\varphi(y)=\begin{cases}
{1, \quad \textrm{if}  \quad  {y}\leq 1 }\\
{0, \quad \textrm{if}  \quad  {y}\geq 2}.
\end{cases}
\end{equation*}
First, for $A\geq 1$ define
$$
\varphi_A(y)=\varphi\left(\frac{y}{A}\right).
$$
Then $\varphi_A(y)=1$ for $y \leq A$ and $\varphi_A(y)=0$ for $y \geq 2A$.
Note that
$$
(\varphi_A)^\prime(y)=\frac{1}{A}\varphi^\prime\left(\frac{y}{A}\right).
$$
Next, define a function
\begin{equation}\label{def-F}
F(y)=\int_{-\infty}^{y}\Lambda Q (z)dz.
\end{equation}
From the properties of $Q$, see \eqref{prop-Q}, there exist $c, \delta>0$ such that
$$
|F(y)|\leq c \int_{-\infty}^{y}e^{-\frac{\delta}{2}|z|}dz,
$$
and thus,
$$
|F(y)|\leq c\, e^{\frac{\delta}{2}y} \quad \mbox{for} \quad y<0.
$$
Moreover, $F$ is a bounded function on all of $\R$.

Now, we define the functional
\begin{equation}\label{def-JA}
J_A(s)=\int_{\R}\ep(s)F(y)\varphi_A(y) \, dy.
\end{equation}
The nontruncated version of this virial-type functional appeared in \cite[Section 3]{MM-KdV-instability}.
The strategy in \cite{MM-KdV-instability} was to show that a virial-type functional is well-defined, bounded above, however, its derivative is bounded from below by a positive constant, which all together led to a contradiction for large times.
We will use the same approach, however, in this proof we will avoid using pointwise decay estimates as it was done in \cite{MM-KdV-instability}. Truncation helps in a straightforward way to get an upper bound on the functional. The bound on the derivative from below turns out to be impossible to get only via truncation, and we use the monotonicity. On the other hand, if we remove the truncation, then it is possible to control virial-type functional only with monotonicity, which we show in the last section. We include both proofs to illustrate different approaches.

It is clear that $J_A(s)$ is well-defined if $\ep(s)\in L^2(\R)$ due to the truncation. Furthermore, $J_A$ is upper bounded by a constant depending on $A$ and $\| \ep \|_2$. Indeed, from the definition of $\varphi_A$, we deduce that
\begin{align}\label{Bound-JA}
\nonumber
|J_A(s)|\leq &\int_{y\leq 0}\left|\ep(s)F(y)\right| dy+ \int_{0}^{2A}\left|\ep(s)F(y)\right|dy\\
\nonumber
\leq & c \, \|\ep(s)\|_2\left(\int_{y\leq 0} e^{\delta y}dy \right)^{1/2}\!\!\!\!\! + c \, A^{1/2}\|F\|_{\infty}\left(\int_{0}^{2A}\left|\ep(s)\right|^2dy\right)^{1/2}\\
\leq & c \, (1+A^{1/2})\|\ep(s)\|_2.
\end{align}
In the next lemma we compute the derivative $\frac{d}{ds}J_A(s)$.
\begin{lemma}\label{J'_A}
Suppose that $\ep(s)\in H^1(\R)$ for all $s\geq 0$ and $\|\ep(s)\|_{H^1}\leq 1$. Then the function $s\mapsto J_A(s)$ is $C^1$ and
$$
\frac{d}{ds}J_A=-\frac{\la_s}{2\la}\left(J_A-\frac14\left(\int Q\right)^2\right)+2\left(1-\frac{1}{4}\left(\frac{x_s}{\la}-1\right)\right)\int \ep Q + R(\ep, A),
$$
where there exists a universal constant $C_3>0$ such that for all $A\geq 1$ we have
\begin{align}\label{R}
|R(\ep,A)|\leq & C_3 \left( \|\ep\|^2_2+\|\ep\|^2_2\|\ep\|_{H^1}+A^{-1/2}\|\ep\|_{L^2(y\geq A)}\right.\nonumber\\
&\left.+\left|\frac{x_s}{\lambda}-1\right|(A^{-1}+\|\ep\|_2)+\left|\frac{\lambda_s}{\lambda}\right|(A^{-1}+A^{1/2}\|\ep\|_{L^2(y\geq A)})\right).
\end{align}

\end{lemma}
\begin{proof}
First, from Lemma \ref{eq-ep}, we have,
\begin{align*}
\frac{d}{ds}J_A= & \int \ep_sF\varphi_A \nonumber\\
=&\int \left((L\ep)_{y}+\frac{\la_s}{\la} \Lambda \ep +\left(\frac{x_s}{\la} -1 \right) \ep_{y}\right)F\varphi_A\\
& +\int  \left(\frac{\la_s}{\la} \Lambda Q +\left(\frac{x_s}{\la} -1 \right) Q_{y}\right)F\varphi_A\\
&-\int  R(\ep)_{y}F\varphi_A\\
&\equiv (I)+(II)+(III),
\end{align*}
where $R(\ep)$ is given by relation \eqref{R(ep)}.

We start with estimating the last term. Since $\|\varphi_A\|_{\infty}\leq 1$ and $\varphi_{y}\in L^{\infty}(\cR)$, integrating by parts, we obtain
\begin{align}\label{III}
(III)= & \int R(\ep) \Lambda Q \varphi_A+\int R(\ep) F \frac{1}{A}\varphi^\prime\left(\frac{y}{A}\right) \nonumber\\
\leq &\|\Lambda Q\|_{\infty}\int |R(\ep)|+\frac{\|F\|_{\infty}\|\varphi^\prime\|_{\infty}}{A}\int |R(\ep)|\nonumber\\
\leq & c_0\left(\|\Lambda Q\|_{\infty}+ \frac{\|F\|_{\infty}\|\varphi^\prime\|_{\infty}}{A}\right)\left( \|\ep\|^2_2+\|\ep\|^2_2\|\ep\|_{H^1}\right),
\end{align}
where in the last line we have used the Gagliardo-Nirenberg inequality and the fact that $\|\ep(s)\|_{\infty}\leq \|\ep(s)\|_{H^1}\leq 1$. Moreover, $c_0>0$ is independent of $A\geq 1$.

Dealing with the second term $(II)$, we get
\begin{align*}
(II)= & \frac{\la_s}{\la} \int \Lambda QF\varphi_A +\left(\frac{x_s}{\la} -1 \right) \int Q_{y}F\varphi_A\\
\equiv &\frac{\la_s}{\la} (II.1)+\left(\frac{x_s}{\la} -1 \right)(II.2),
\end{align*}
where, due to $F_{y}=\Lambda Q$, see \eqref{def-F}, yields
\begin{align*}
(II.1)= & \frac12\int (F^2)_{y}\varphi_A\\
= &\frac12\int (F^2)_{y}+\frac12\int (F^2)_{y}(\varphi_A-1)\\
= & \frac12\left(\int \Lambda Q\right)^2+\frac12\int (F^2)_{y}(\varphi_A-1)\\
\equiv & \frac12\left(\int \Lambda Q\right)^2+R_1(A).
\end{align*}
Now, integration by parts gives
\begin{align*}
\int \Lambda Q dy=&\frac12\int Q dy+\int y Q_{y} dy=-\frac12\int Q,\\
\end{align*}
and therefore,
$$
\frac12\left(\int \Lambda Q\right)^2=\frac18\left(\int Q\right)^2<+\infty.
$$
Moreover, since $\|\varphi_A-1\|_{\infty}\leq 2$ and
$$
\supp(\varphi_A-1)\subset \left\{y \in \R: y\geq A\right\},
$$
we have
\begin{align}\label{R_1}
|R_1(A)|\leq& \int_{{y}\geq A} |\Lambda Q F|\leq \|F\|_{\infty}\int_{{y}\geq A} |\Lambda Q|\frac{|{y}|}{|{y}|} \nonumber\\
\leq& \frac{\|F\|_{\infty}\|y\Lambda Q\|_1}{A}.
\end{align}
Note that $\| y \Lambda Q \|_1 \leq 3 \, \|y Q \|_2 = \it{const}$.

On the other hand, since $\Lambda Q \perp Q$ (we are in the critical case)
\begin{align*}
(II.2)= & -\int Q\Lambda Q \varphi_A-\int Q F \frac{1}{A}\varphi^\prime\left(\frac{y}{A}\right)\\
= &-\int Q\Lambda Q-\int Q\Lambda Q(\varphi_A-1)-\int Q F \frac{1}{A}\varphi^\prime\left(\frac{y}{A}\right)\\
\equiv & \, R_2(A),
\end{align*}
where, since $\|\varphi_A-1\|_{\infty}\leq 2$, $\varphi^\prime\in L^{\infty}$ and
$\textrm{supp}(\varphi_A-1)\subset \left\{y\in \R:{y}\geq A\right\}$, we obtain that
\begin{align}\label{R_2}
|R_2(A)|\leq& \int_{{y}\geq A} |Q\Lambda Q |+\|F\|_{\infty}\|\varphi^\prime\|_{\infty}\frac{1}{A}\int |Q|\nonumber \\
\leq&  \int_{{y}\geq A} |Q\Lambda Q |\frac{|{y}|}{|{y}|}+\frac{\|F\|_{\infty}\|\varphi^\prime\|_{\infty}\|Q\|_{1}}{A}\nonumber \\
\leq & \frac{1}{A}\left(\|\Lambda Q\|_{2}\|yQ\|_2+\|F\|_{\infty}\|\varphi^\prime\|_{\infty}\|Q\|_{1}\right) \equiv \frac{c}{A},
\end{align}
with the constant $c$ independent of A and $\ep$.

Next we estimate the term $(I)$. Applying integration by parts, we get
\begin{align*}
(I)= & -\int (L\ep)\Lambda Q \varphi_A-\int (L\ep)F\frac{1}{A}\varphi^\prime\left(\frac{y}{A}\right)\\
&+\frac{\la_s}{\la}\int \Lambda \ep F \varphi_A \\
&-\left(\frac{x_s}{\la} -1 \right)\left(\int \ep \Lambda Q \varphi_A+\int \ep F \frac{1}{A}\varphi^\prime\left(\frac{y}{A}\right)\right)\\
\equiv &(I.1)+\frac{\la_s}{\la} (I.2)-\left(\frac{x_s}{\la} -1 \right)(I.3).
\end{align*}
Let us first consider the term $(I.3)$. Using the definition \eqref{Eq:LambdaQ}, we have
\begin{align*}
(I.3)= &\frac12\int \ep  Q + \frac12\int \ep  Q (\varphi_A-1)+\int \ep  yQ_{y} \varphi_A+\int \ep F \frac{1}{A}\varphi^\prime\left(\frac{y}{A}\right)\\
\equiv &\frac12\int \ep  Q+R_3(\ep,A),
\end{align*}
where
\begin{align}\label{R_3}
|R_3(\ep,A)|\leq &\int_{\{y}\geq A\} \ep  Q \frac{|{y}|}{|{y}|}+\|yQ_{y}\|_2\|\ep\|_2
+\frac{\|F\|_{\infty}\|\varphi^\prime\|_{\infty}}{A}\left(\int_{A\leq y\leq 2A} |\ep|dy\right)\nonumber \\
\leq & \, c \left(1+\frac{1}{A}+\frac{1}{A^{1/2}}\right)\|\ep\|_2
\end{align}
with the constant $c>0$ independent of $\ep$ and $A$.

Next, we turn to the term $(I.2)$. Integration by parts yields
 \begin{align*}
(I.2)= &\frac12\int  \ep F \varphi_A+\int  y\ep_{y} F \varphi_A\\
=& \frac12\int  \ep F \varphi_A- \int  \ep F \varphi_A-\int y\ep \Lambda Q \varphi_A-\int y\ep F \frac{1}{A}\varphi^\prime\left(\frac{y}{A}\right)\\
\equiv& -\frac12 J_A + R_4(\ep,A),
\end{align*}
where in the last line we used the definition \eqref{def-JA}. We estimate $R_4(\ep,A)$. Indeed, it is clear that
$$
\int y\ep \Lambda Q \varphi_A \leq \|y\Lambda Q\|_2\|\ep\|_2.
$$
Moreover,
\begin{align*}
\int y\ep F \frac{1}{A}\varphi^\prime\left(\frac{y}{A}\right)\leq & \frac{1}{A}\|F\|_{\infty}\|\varphi^\prime\|_{\infty}\int_{A\leq |y|\leq 2A} |y\ep|\\
\leq & \frac{1}{A}\|F\|_{\infty}\|\varphi^\prime\|_{\infty}\|\ep\|_{L^2(y\geq A)}\left(\int_{A\leq |y|\leq 2A}|y|^2\right)^{1/2}\\
\leq & 4A^{1/2}\|F\|_{\infty}\|\varphi^\prime\|_{\infty}\|\ep\|_{L^2(y\geq A)}.
\end{align*}
Collecting the last two estimates, we deduce
\begin{equation}\label{R_4}
|R_4(\ep,A)|\leq c(\|\ep\|_2+A^{1/2}\|\ep\|_{L^2(y\geq A)}),
\end{equation}
where $c>0$ is again independent of $\ep$ and $A$.

To estimate $(I.1)$ we recall the definition of the operator $L$ to deduce
\begin{align*}
L(fg) =&-(fg)_{xx}+fg-5Q^4fg \\
= & -f_{xx}g-2f_xg_x-fg_{xx}+fg-5Q^4fg\\
=&(Lf)g-2f_{x}g_{x}-fg_{xx}.
\end{align*}
Therefore,
\begin{align*}
L(\Lambda Q \varphi_A) =&(L\Lambda Q)\varphi_A-2(\Lambda Q)_{y}\frac{1}{A}\varphi^\prime\left(\frac{y}{A}\right)-\Lambda Q\left(\frac{1}{A^2}\varphi^{\prime \prime}\left(\frac{y}{A}\right)\right)\\
\equiv& L(\Lambda Q)\varphi_A+G_A
\end{align*}
and
\begin{align*}
L\left(F\frac{1}{A}\varphi^\prime\left(\frac{y}{A}\right)\right) =&\frac{1}{A}\left[L\left(\varphi^\prime\left(\frac{y}{A}\right)\right)F-2\Lambda Q \frac{1}{A}\varphi^{\prime \prime}\left(\frac{y}{A}\right)-(\Lambda Q)_y\varphi^\prime\right]\\
\equiv & \frac{1}{A} H_A.
\end{align*}
So using that $L$ is a self-adjoint operator and $L(\Lambda Q)=-2Q$, we get
\begin{align*}
(I.1)= &-\int  \ep L(\Lambda Q \varphi_A)-\int \ep L\left(F\frac{1}{A}\varphi_{y}\left(\frac{y}{A}\right)\right)\\
=& 2\int \ep Q \varphi_A-\int \ep\left(G_A + \frac{1}{A}H_A\right)\\
=& 2\int \ep Q+ 2\int \ep Q(\varphi_A-1)-\int \ep\left(G_A + \frac{1}{A}H_A\right)\\
\equiv & 2\int \ep Q + R_5(\ep, A).
\end{align*}
We estimate the terms in $R_5(\ep, A)$ separately. First, note that
$$
\int\ep Q(\varphi_A-1)\leq 2\int_{\{y}\geq A\} |\ep Q|\frac{|{y}|}{|{y}|}\leq \frac{2}{A}\|yQ\|_2\|\ep\|_2.
$$
Next,
\begin{align*}
\int \ep G_A \leq & \frac{2}{A}\|\varphi^\prime\|_{\infty}\|(\Lambda Q)_{y}\|_2\|\ep\|_2 + \frac{1}{A^2}\|\varphi^{\prime \prime}\|_{\infty}\|\Lambda Q\|_2\|\ep\|_2.
\end{align*}
Now observe that $\|H_A\|_{\infty}\leq c$ (independent of $A\geq 1$) and $\supp(H_A)\subset \left\{A\leq y\leq 2A\right\}$, thus,
$$
\frac{1}{A}\int \ep H_A= \frac{1}{A}\int_{A\leq y\leq 2A} \ep H_A\leq \frac{1}{A}\|H_A\|_{\infty}\left(\int_{A\leq y\leq 2A}|\ep| dy\right)\leq \frac{2\|H_A\|_{\infty}}{A^{1/2}}\|\ep\|_2.
$$
Hence, for $A\geq 1$
\begin{equation}\label{R_5}
|R_5(\ep, A)|\leq \frac{c}{A^{1/2}}\|\ep\|_{L^2(y\geq A)},
\end{equation}
where once again $c>0$ is independent of $\ep$ and $A$.

Collecting all above estimates, we obtain
\begin{align*}
\frac{d}{ds}J_A= & 2\int \ep Q + R_5(\ep, A)+\frac{\la_s}{\la} \left(-\frac12J_A + R_4(\ep,A)\right)-\left(\frac{x_s}{\la} -1 \right)\left(\frac12\int \ep  Q+R_3(\ep,A)\right)\\
& +\frac{\la_s}{\la} \left(\frac18\left(\int Q\right)^2 +R_1(A)\right)+\left(\frac{x_s}{\la} -1 \right)R_2(A)\\
&+(III)\\
=& -\frac{\la_s}{2\la}\left(J_A-\kappa
\right)+2\left(1-\frac{1}{4}\left(\frac{x_s}{\la}-1\right)\right)\int \ep Q + R(\ep, A),
\end{align*}
where
\begin{equation}\label{kappa}
\kappa=\frac14\left(\int Q\right)^2
\end{equation}
and
$$
R(\ep, A)= (III)+R_5(\ep,A)+\left(\frac{x_s}{\la}-1\right)\left(R_2(A)-R_3(\ep, A)\right)+\frac{\la_s}{\la}\left(R_1(A)+R_4(\ep, A)\right).
$$
Finally, there exists a universal constant $C_3>0$ (independent of $\ep$ and $A$) such that, in view of \eqref{III}, \eqref{R_1}, \eqref{R_2}, \eqref{R_3}, \eqref{R_4} and \eqref{R_5}, for all $A\geq 1$ the inequality \eqref{R} holds.
\end{proof}

\begin{lemma}[Comparison between $M_0$, $\ep_0$ and $\int \ep_0Q$]\label{lemma-M_0}
There exists a universal constant $C_4>0$ such that for $\|\ep_0\|_{H^1}\leq 1$, we have
$$
\left|M_0-2\int\ep_0Q\right|+\left|E_0+\int\ep_0Q\right|+\left|E_0+\frac12 M_0\right|\leq C_4 \|\ep_0\|^2_{H^1}.
$$
\end{lemma}
\begin{proof}
First, observe that from the definition \eqref{M_0}, we have
$$
M_0-2\int\ep_0Q=\int\ep_0^2,
$$
and so $\left|M_0-2\int\ep_0Q\right|=\|\ep_0\|^2_{2}$.
Next, from \eqref{E-lin} we deduce
\begin{align*}
E_0=E[Q+\ep_0]=&\frac12 \left( L \ep_0, \ep_0 \right) - \frac{1}{2}M_0\\
& - \frac16 \left( 20 \int Q^3\ep_0^3 +  15\int Q^2\ep_0^4+6\int Q\ep_0^5+\int\ep_0^6\right),
\end{align*}
which implies, for some universal constant $c>0$, that
$$
\left|E_0+\frac12 M_0\right|\leq c\|\ep_0\|^2_{H^1},
$$
by the definition of $L$, and the fact that $\|\ep_0\|_{\infty}\leq \|\ep_0\|_{H^1}\leq 1$.

Finally,
$$
\left|E_0+\int\ep_0Q\right|\leq \left|E_0+\frac12 M_0\right|+\frac{1}{2}\left|M_0-2\int\ep_0Q\right|\leq \left(c+\frac{1}{2}\right) \|\ep_0\|^2_{H^1},
$$
and setting $C_4=c+\frac{1}{2}$ concludes the proof.
\end{proof}

\begin{lemma}[Control of $\|\ep(s)\|_{H^1}$]\label{H^1-control}
There exists $0<\alpha_2<1$ such that if $\|\ep(s)\|_{H^1}<\alpha$, $|\lambda(s)-1|<\alpha$, where $\alpha<\alpha_2$, and $\ep(s)\perp \{Q_{y}, Q^3\}$ for all $s\geq 0$, then there exists a universal constant $C_5>0$ such that
$$
\left(L\ep(s), \ep(s)\right)\leq \|\ep(s)\|^2_{H^1}\leq C_5\left(\alpha \left|\int \ep_0 Q\right|+\|\ep_0\|^2_{H^1}\right).
$$
\end{lemma}

\begin{proof}
This is Lemma 11 in Martel-Merle \cite{MM-KdV-instability}, however, we provide a proof here, since our statement differs slightly from the one in \cite{MM-KdV-instability}. From \eqref{E-lin} we have
\begin{align*}
\left( L \ep(s), \ep(s) \right)=& 2E[Q+\ep(s)] + M_0\\
& - \frac16 \left( 20 \int Q^3\ep^3 +  15\int Q^2\ep^4+6\int Q\ep^5+\int\ep^6\right).
\end{align*}
Since $\|\ep(s)\|_{\infty}\leq \|\ep(s)\|_{H^1}\leq 1$, there exists a universal constant $c>0$ such that
\begin{align}\label{Coerc}
\left( L \ep(s), \ep(s) \right)\leq & 2E[Q+\ep(s)] + M_0 + c\, \|\ep(s)\|_{H^1}\|\ep(s)\|^2_{2} \nonumber\\
\leq &2E[Q+\ep(s)] + M_0 + c\, \|\ep(s)\|_{H^1}\left( L \ep(s), \ep(s) \right),
\end{align}
where in the last line we have used the coercivity of the quadratic form $(L\cdot, \cdot)$, provided $\ep(s)\perp \{Q_{y}, Q^3\}$, see  Lemma \ref{Lemma-ort3}.

Now, there exists $\alpha_2>0$ such that if $\|\ep(s)\|_{H^1}<\alpha$ for all $s\geq 0$, where $\alpha<\alpha_2$, then
$$
c \, \|\ep(s)\|_{H^1}\leq \frac{1}{2}.
$$
Therefore, the last term on the right hand side of \eqref{Coerc} may be absorbed into the left side, and we obtain
\begin{align*}
\left( L \ep(s), \ep(s) \right)\leq & 4E[Q+\ep(s)] + 2M_0  \nonumber\\
\leq &4\lambda^2(s)E_0 + 2M_0,
\end{align*}
where in the last line we have used the second relation in \eqref{ep-M+E}.

Next, we use the last estimate to control the $H^1$-norm of $\ep(s)$. Indeed, from the definition of $L$, we have
\begin{align*}
\|\ep(s)\|^2_{H^1}=&\int \ep^2(s)+\int|\ep_y(s)|^2=\left( L \ep(s), \ep(s) \right)+ 5\int Q^4\ep^2(s)\\
\leq & \left( L \ep(s), \ep(s) \right)+ \|5 Q^4 \|_{\infty}\|\ep(s)\|^2_2\\
\leq &\left(1+{\|5Q^4\|_{\infty}}\right)\left( L \ep(s), \ep(s) \right)\\
\leq &\left(1+\|5Q^4\|_{\infty}\right)(4\lambda^2(s)E_0 + 2M_0)\\
\leq &4\left(1+\|5Q^4\|_{\infty}\right)\left((\lambda(s)-1)(\lambda(s)+1)|E_0| + \left|E_0+\frac12M_0\right|\right).
\end{align*}

Finally, since $|\lambda(s)-1|<\alpha$, for $\alpha <1$ we have $|\lambda(s)+1|\leq 3$, and applying Lemma \ref{lemma-M_0}, we deduce
\begin{align*}
(\lambda(s)-1)(\lambda(s)+1)|E_0|+\left|E_0+\frac12M_0\right|\leq &3 \alpha \left(\left|E_0+\int \ep_0 Q\right|+\left|\int \ep_0 Q\right|\right)+C_4 \|\ep_0\|^2_{H^1}\\
\leq & 3 \alpha \left(C_4 \|\ep_0\|^2_{H^1}+\left|\int \ep_0 Q\right|\right)+C_4 \|\ep_0\|^2_{H^1},
\end{align*}
which implies, since $\alpha<1$, the existence of an universal constant $C_5>0$ such that
$$
\left( L \ep(s), \ep(s) \right)\leq \|\ep(s)\|^2_{H^1} \leq C_5\left(\alpha \left|\int \ep_0 Q\right|+\|\ep_0\|^2_{H^1}\right),
$$
which concludes the proof.
\end{proof}

\section{The proof of $H^1$-instability of $Q$} \label{S-5}

In this section we prove Theorem \ref{Theo-Inst}. Let $0<b_0<1$ to be chosen later and set the initial data
$$
u_0=Q+\ep_0
$$
with $\|\ep_0\|_{H^1}\leq 1$ satisfying
\begin{equation}\label{ep_0}
\|\ep_0\|^2_{H^1}\leq b_0 \int\ep_0Q\leq b_0\|Q\|^2.
\end{equation}

Moreover, we can also assume that
\begin{equation}\label{ep_02}
\ep_0=u_0-Q \perp \{Q_{y}, Q^3\}
\end{equation}
and, for all $y\in \R$
\begin{equation}\label{ep_03}
|\ep_0(y)|\leq \tilde{c}\, e^{-\tilde{\delta}|y|} \textrm{ for some } \tilde{c}>0 \textrm{ and } \tilde{\delta}>0.
\end{equation}

\begin{remark}
One simple example of such initial data is (for all $n\in \mathbb{N}$)
$$
\ep^n_0=\frac{1}{n}(Q+rQ^3),
$$
where $r\in \R$ is chosen such that $\ep^n_0 \perp Q^3$, that is, $r=\displaystyle -\frac{\int Q^4}{\int Q^6}$.
\end{remark}

Assume, by contradiction, that $Q$ is stable. Therefore, for  $\alpha_0<\overline{\alpha}$, where $\overline{\alpha}>0$ is given by Proposition \ref{ModThI}, if $b_0$ is sufficiently small, we have $u(t)\in U_{\alpha_0}$ (recall \eqref{tube}). Thus, from Definition \ref{eps}, there exist functions $\lambda(t)$ and $x(t)$ such that
\begin{equation*}
\ep(t) =\lambda^{1/2}(t) \, u( \lambda(t)y +x(t)) - Q(y) \perp \{Q_{y},  Q^3\}
\end{equation*}
and also $\lambda(0)=1$ and $x(0)=0$ (by \eqref{ep_02}).

Next rescaling time $t \mapsto s$ by $\frac{ds}{dt} = \frac1{\la^3}$ and taking $\alpha_0<\alpha_1$, where $\alpha_1>0$ is given by Lemma \ref{Lemma-param}, we obtain that $\lambda(s)$ and $x(s)$ are $C^1$ functions and $\ep(s)$ satisfies the equation \eqref{ep1-order1}. Moreover, from Proposition \ref{ModThI}, since $u(t)\in U_{\alpha_0}$, we have
\begin{equation}\label{Bound_ep}
\|\ep(s)\|_{H^1}\leq C_1 \alpha_0 \quad \textrm{and} \quad |\lambda(s)-1|\leq C_1\alpha_0.
\end{equation}
Furthermore, in view of \eqref{ControlParam}, if $\alpha_0>0$ is small enough, we deduce that
$$
\left|\frac{\lambda_s}{\lambda}\right|+\left|\frac{x_s}{\lambda}-1\right|\leq C_2\|\ep(s)\|_2\leq C_1C_2\alpha_0.
$$
Since $ x_t=x_s/\lambda^3$, we conclude that
$$
\frac{1-C_1C_2\alpha_0}{(1+C_1\alpha_0)^2}\leq \frac{1-C_1C_2\alpha_0}{\lambda^2}\leq x_t\leq \frac{1+C_1C_2\alpha_0}{\lambda^2}\leq \frac{1+C_1C_2\alpha_0}{(1-C_1\alpha_0)^2}.
$$
Therefore, we can choose $\alpha_0>0$ small enough such that
$$
\frac{3}{4}\leq x_t \leq \frac{5}{4}.
$$
The last inequality implies that $x(t)$ is increasing and by the Mean Value Theorem
$$
x(t_0)-x(t)\geq \frac{3}{4}(t_0-t)
$$
for every $t_0, t\geq 0$ with $t\in [0,t_0]$.
Also, recalling $x(0)=0$, another application of the Mean Value Theorem yields
$$
x(t)\geq \frac{1}{2}t
$$
for all $t\geq 0$. Finally, by assumption \eqref{ep_03} and properties of $Q$, we have
$$
|u_0(x)|\leq ce^{-{\delta}|x|}
$$
for some $c>0$ and ${\delta}>0$. Hence, from Lemma \ref{AM2}, for a fixed $M\geq \max\{4, \frac{1}{\delta}\}$, there exists $C=C(M,\delta)>0$ such that for all $t\geq 0$ and $x_0>0$, we have
$$
\int_{x>x_0}u^2(t, x+x(t))dx\leq Ce^{-\frac{x_0}{M}}.
$$

The next result provides $L^2$ exponential decay on the right also for $\ep(s)$.
\begin{corollary}\label{AM3}
Let $M\geq \max\{4, \frac{1}{\delta}\}$. If $\alpha_0>0$ is sufficiently small, then there exists $C=C(M,\delta)>0$ such that for every $s\geq 0$ and $y_0>0$
$$
\int_{y>y_0}\ep^2(s, y)dy\leq Ce^{-\frac{y_0}{2M}}.
$$
\end{corollary}
\begin{proof}
From the definition of $\ep(s)$, we have
$$
\frac{1}{\lambda^{1/2}(s)}\ep\left(s,\frac{y}{\lambda(s)}\right)=u(s,y+x(s))-\frac{1}{\lambda^{1/2}(s)}Q\left(\frac{y}{\lambda(s)}\right).
$$
Recalling the property of $Q$ (and that $M \geq 4$), we get
\begin{equation}\label{Q-decay}
\frac{1}{\lambda^{1/2}(s)}Q\left(\frac{y}{\lambda(s)}\right)\leq \frac{c}{\lambda^{1/2}(s)}e^{-\left|\frac{y}{\lambda(s)}\right|}\leq \sqrt{2}ce^{-\frac{2}{3}|x|}\leq ce^{-\frac{|x|}{M}},
\end{equation}
and if $\alpha_0<(2C_1)^{-1}$, then $1/2\leq \lambda(s)\leq 3/2$.
Therefore, using Lemma \ref{AM2} and \eqref{Q-decay}, we deduce that
\begin{align*}
\int_{y>y_0}\frac{1}{\lambda(s)}\ep^2\left(s,\frac{y}{\lambda(s)}\right)dy\leq &2\int_{y>y_0}u^2(s,y+x(s))dy+2\int_{y>y_0}\frac{1}{\lambda(s)}Q^2\left(\frac{y}{\lambda(s)}\right)\\
\leq & 2c\, e^{-\frac{y_0}{M}}+2c\int_{y>y_0}e^{-\frac{2y}{M}}dy\\
\leq & c\, e^{-\frac{y_0}{M}}.
\end{align*}
Finally, by the scaling invariance of the $L^2$-norm
$$
\int_{y>y_0}\ep^2(s, y)dy=\int_{y>\lambda(s)y_0}\frac{1}{\lambda(s)}\ep^2\left(s,\frac{y}{\lambda(s)}\right)dy\leq c \, e^{-\frac{\lambda(s)y_0}{M}}\leq C \, e^{-\frac{y_0}{2M}},
$$
since $\lambda(s)\geq 1/2$.
\end{proof}

Next, as in Martel-Merle \cite{MM-KdV-instability} (see also Bona-Souganidis-Strauss \cite{BSS} and Grillakis-Shatah-Strauss \cite{GSS}), we define a rescaled version of $J_A$, which plays a central role in the proof \footnote{Note that we do not make use of the functional $I(s)=\frac12\int y\ep^2(s)$ introduced by Martel-Merle \cite{MM-KdV-instability} in their instability proof.}. Recall the definition of $J_A$ in \eqref{def-JA} and let
$$
K_A(s)=\lambda^{1/2}(s)(J_A(s)-\kappa),
$$
where $\kappa$ is given by \eqref{kappa}.

Therefore, from \eqref{Bound-JA} and \eqref{Bound_ep}, it is clear that
\begin{equation}\label{Bound-KA}
|K_A(s)|\leq \left(\frac32\right)^{1/2} \left(c\,(1+A^{1/2})\|F\|_{\infty}\|\ep(s)\|_2+\kappa\right)<+\infty,
\end{equation}
for all $s\geq 0$.

Moreover, using Lemma \ref{J'_A}, we also have
\begin{align}\label{K'_A}
\frac{d}{ds}K_A =& \frac12\frac{\lambda_s}{\la^{1/2}}\left(J_A-\kappa\right)+ \lambda^{1/2} \frac{d}{ds}J_A \nonumber\\
=& \lambda^{1/2} \left(\frac{d}{ds}J_A+\frac{\lambda_s}{2\lambda}\left(J_A-\kappa\right)\right) \nonumber \\
=&\lambda^{1/2} \left(2\left(1-\frac14\left(\frac{x_s}{\lambda}-1\right)\right)\int\ep Q +R(\ep,A)\right).
\end{align}

In the next result we obtain a strictly positive lower bound for $\ds \frac{d}{ds}K_A(s)$ for a certain choice of $\alpha_0, b_0$ and $A$.
\begin{theorem}
There exist $b_0, \alpha_0>0$ sufficiently small and $A\geq 1$ sufficiently large such that
\begin{equation}\label{Bound-K'_A}
\frac{d}{ds}K_A(s)\geq \frac12 \int \ep_0 Q>0 \quad \textrm{for all} \quad s\geq 0.
\end{equation}
\end{theorem}

\begin{proof}
In view of \eqref{Bound_ep}, let $\alpha_0<\min\{\alpha_1(C_1)^{-1}, \alpha_2(C_1)^{-1}, (2C_1)^{-1}, (2C_1C_2)^{-1}, 1/2\}$, so that we can apply Lemmas \ref{Lemma-param} and \ref{H^1-control}. From \eqref{K'_A} and the definition of $M_0$, see \eqref{M_0}, we have
\begin{equation}\label{K'AM_0}
\frac{d}{ds}K_A(s) = \lambda^{1/2} \left(\left(1-\frac14\left(\frac{x_s}{\lambda}-1\right)\right)M_0 +\widetilde{R}(\ep,A)\right),
\end{equation}
where
$$
\widetilde{R}(\ep,A)={R}(\ep,A)-  \left(1-\frac14\left(\frac{x_s}{\lambda}-1\right)\right) \int \ep^2.
$$
Since $\alpha_0<\min\{(2C_1)^{-1}, (2C_1C_2)^{-1}\}$, we have $1/2\leq \lambda(s)\leq 3/2$, and using \eqref{ControlParam}, we deduce
$$
\lambda^{1/2} \left(1-\frac14\left(\frac{x_s}{\lambda}-1\right)\right)\geq \frac{1}{\sqrt{2}} \cdot  \frac34>\frac12.
$$
Moreover, from the definition of $M_0$, we also obtain
$$
M_0=2\int \ep_0 Q +\int \ep_0^2\geq 2\int \ep_0 Q.
$$
Hence, the first term in \eqref{K'AM_0} is at least
\begin{equation}\label{M_0term}
\lambda^{1/2} \left(1-\frac14\left(\frac{x_s}{\lambda}-1\right)\right)M_0\geq  \int \ep_0Q.
\end{equation}
Now we estimate the second term in \eqref{K'AM_0}. By Lemma \ref{Lemma-param}, we have
$$
\left|\frac{\lambda_s}{\lambda}\right|+\left|\frac{x_s}{\lambda}-1\right|\leq C_2\|\ep(s)\|_2.
$$
Therefore, using the inequalities \eqref{R} and \eqref{Bound_ep}, there exists a universal constant $C_6>0$, such that for $A\geq 1$ we can upper bound the second term as follows
\begin{equation}\label{Rtilda}
\lambda^{1/2} \widetilde{R}(\ep, A)\leq C_6\left(\|\ep(s)\|^2_2 +A^{1/2}\|\ep(s)\|_2\|\ep(s)\|_{L^2(y\geq A)}+A^{-1/2}\|\ep(s)\|_2\right)
\end{equation}
Note that we are at the crucial point here, where we can use the smallness of $\ep$ and a large value of $A$ to try to make all three terms in \eqref{Rtilda} small. The issue will be really in the middle term: taking a large $A$ will make $A^{1/2}$ big, while the tail of $\ep$ in $\|\ep(s)\|_{L^2(y\geq A)}$ has to balance the growth of $A^{1/2}$, which is delicate and only possible to get via monotonicity (also, with pointwise decay estimates as in the original proof in \cite{MM-KdV-instability}, but we are trying to avoid that here).

Continuing, by Lemma \ref{H^1-control} we get
$$
\|\ep(s)\|^2_{H^1}\leq C_5\left(C_1\alpha_0 \left|\int \ep_0 Q\right|+\|\ep_0\|^2_{H^1}\right),
$$
and thus, the assumption \eqref{ep_0} yields
\begin{align*}
\|\ep(s)\|^2_{H^1}\leq& C_5\left(C_1\alpha_0 \int \ep_0 Q+b_0\int \ep_0 Q \right)\nonumber \\
\leq & C_5(C_1+1) \left(\alpha_0+b_0\right)\int \ep_0 Q.
\end{align*}
From Corollary \ref{AM3}, we obtain the control on the tail of $\ep$ in $L^2$, hence, for every $A\geq 1$ we have
\begin{align}\label{ep(s)2-decay}
\|\ep(s)\|^2_{L^2(y\geq A)}\leq C_7e^{-\frac{A}{2M}}
\end{align}
for some constants $C_7>0$ and $M>0$.

Now, collecting estimates \eqref{Rtilda}-\eqref{ep(s)2-decay}, there exists $C_8>0$ such that
\begin{align*}
\lambda^{1/2} \widetilde{R}(\ep, A)\leq & C_8\left(\left(\alpha_0+b_0\right)\int \ep_0 Q +(A^{1/2}e^{-\frac{A}{4M}}+A^{-1/2})(\alpha_0+b_0)^{1/2}\left(\int \ep_0 Q\right)^{1/2}\right)\\
\leq & C_8(\alpha_0+b_0)^{1/2}\left(\int \ep_0 Q +(A^{1/2}e^{-\frac{A}{4M}}+A^{-1/2})\left(\int \ep_0 Q\right)^{1/2}\right),
\end{align*}
where in the last inequality we assume that $\alpha_0+b_0<1$, since these numbers will be chosen small enough in the sequel.

Now, let $\alpha_0, b_0>0$ sufficiently small such that  $C_8(\alpha_0+b_0)^{1/2}\leq 1/4$, and then (for $\alpha_0$ and $b_0$ fixed) let $A\geq 1$ be large enough such that
$$
(A^{1/2}e^{-\frac{A}{4M}}+A^{-1/2})\leq \left(\int \ep_0 Q\right)^{1/2},
$$
this is exactly where monotonicity helps to play against the truncation and make the term $A^{1/2} \, e^{-\frac{A}{4M}}$ not just bounded, but small.

Thus,
$$
\lambda^{1/2} \widetilde{R}(\ep, A)\leq 2C_8(\alpha_0+b_0)^{1/2}\int \ep_0 Q \leq \frac12\int \ep_0 Q,
$$
which implies, in view of \eqref{K'AM_0} and \eqref{M_0term}, that
$$
\frac{d}{ds}K_A(s)\geq \frac12 \int \ep_0 Q>0 \quad \textrm{for all} \quad s\geq 0.
$$
\end{proof}

Finally, we have all the ingredients to obtain the main result.
\begin{proof}[Proof of Theorem \ref{Theo-Inst}]
Integrating in $s$ variable both sides of inequality \eqref{Bound-K'_A}, we obtain
$$
K_A(s)\geq \frac{s}{2}\int \ep_0Q + K_A(0) \quad \textrm{for all} \quad s\geq 0.
$$
Hence,
$$
\lim_{s\rightarrow \infty}K_A(s)=\infty,
$$
which is a contradiction with \eqref{Bound-KA}, the boundedness of $K_A(s)$ for all $s>0$.
Thus, our original assumption that $Q$ is stable is not valid and this finishes the proof.
\end{proof}

\section{An alternative proof without truncation} \label{S-6}

We define
\begin{equation*}
J(s)=\int_{\R}\ep(s)F(y) \, dy.
\end{equation*}
Since we don't have any truncation now, our first goal is to show that $J(s)$ is well-defined and upper bounded\footnote{Note that with truncation the boundedness of $J_A$ was basically built in. Now we have to prove it, which we do via monotonicity bounds.} for all $s\geq 0$ under a certain choice of initial condition $\ep_0$. To this end we use Monotonicity, see Lemma \ref{AM2}. Indeed, assume that $\ep_0$ satisfies the assumptions \eqref{ep_0}-\eqref{ep_03}. Therefore, we can apply Corollary \ref{AM3} to deduce, for all $s\geq 0$, that
\begin{align}\label{Bound-J}
\nonumber
|J(s)|\leq &\int_{y\leq 0}\left|\ep(s)F(y)\right| dy+ \int_{y>0}\left|\ep(s)F(y)\right|dy\\
\nonumber
\leq &c\|\ep(s)\|_2\left(\int_{y\leq 0} e^{\delta y}dy \right)^{1/2}\!\!\!\!\! + \|F\|_{\infty}\sum_{k=0}^{+\infty}\int_{k}^{k+1}\left|\ep(s)\right|dy\\  \nonumber
\leq & c\|\ep(s)\|_2+\|F\|_{\infty}\sum_{k=0}^{+\infty}\left(\int_{k}^{+\infty}\left|\ep(s)\right|^2dy\right)^{1/2}\\\nonumber
\leq & c\|\ep(s)\|_2+c\|F\|_{\infty}\sum_{k=0}^{+\infty}e^{-\frac{k}{4M}}\\
\leq & c\alpha_0^{1/2}+c\|F\|_{\infty},
\end{align}
where in the last inequality we have used \eqref{Bound_ep}.

Next, we consider the derivative $\ds \frac{d}{ds}J(s)$. Again, from Martel-Merle \cite[Lemma 6]{MM-KdV-instability}, we have
\begin{lemma}\label{J'}
Suppose that $\ep(s)\in H^1(\R)$ with $\|\ep(s)\|_{H^1}\leq 1$ for all $s\geq 0$. Then the function $s\mapsto J(s)$ is $C^1$ and
$$
\frac{d}{ds}J=-\frac{\la_s}{2\la}\left(J-\frac14\left(\int Q\right)^2\right)+2\left(1-\frac{1}{4}\left(\frac{x_s}{\la}-1\right)\right)\int \ep Q + R(\ep),
$$
where there exists a universal constant $c>0$ such that
\begin{align}\label{R-2}
|R(\ep)|\leq & c\left(\|\ep\|^2_2+\|\ep\|^2_2\|\ep\|_{H^1}+\left(\left|\frac{x_s}{\lambda}-1\right|+\left|\frac{\lambda_s}{\lambda}\right|\right)\|\ep\|_2\right).
\end{align}
\end{lemma}

Now define
$$
K(s)=\lambda^{1/2}(s)(J(s)-\kappa),
$$
where $\kappa$ is given by \eqref{kappa}.

Therefore, from \eqref{Bound-J} and \eqref{Bound_ep}, if $\alpha_0$ is sufficiently small, it is clear that
\begin{equation}\label{Bound-K}
|K(s)|\leq \left(\frac32\right)^{1/2}\left(c\alpha_0^{1/2}+c\|F\|_{\infty}+\kappa\right)<+\infty
\end{equation}
for all $s\geq 0$.

Now, if we obtain a strictly positive lower bound for $\ds \frac{d}{ds}K(s)$, we deduce that $\ds \lim_{s\rightarrow \infty}K(s)=\infty$ (as in the first proof of Theorem \ref{Theo-Inst} in Section \ref{S-5}) and reach a contradiction with \eqref{Bound-K}. From Lemma \ref{J'} we have
\begin{align*}
\frac{d}{ds}K =& \frac12\frac{\lambda_s}{\la^{1/2}}\left(J-\kappa\right)+ \lambda^{1/2} \frac{d}{ds}J \nonumber\\
=& \lambda^{1/2} \left(\frac{d}{ds}J+\frac{\lambda_s}{2\lambda}\left(J-\kappa\right)\right) \nonumber \\
=&\lambda^{1/2} \left(2\left(1-\frac14\left(\frac{x_s}{\lambda}-1\right)\right)\int\ep Q +R(\ep)\right).
\end{align*}
Using the definition of $M_0$, see \eqref{M_0}, we have
\begin{equation}\label{K'M_0}
\frac{d}{ds}K(s) = \lambda^{1/2} \left(\left(1-\frac14\left(\frac{x_s}{\lambda}-1\right)\right)M_0 +\widetilde{R}(\ep)\right),
\end{equation}
where
$$
\widetilde{R}(\ep)={R}(\ep)-  \left(1-\frac14\left(\frac{x_s}{\lambda}-1\right)\right) \int \ep^2.
$$
Since $\alpha_0<\min\{(2C_1)^{-1}, (2C_1C_2)^{-1}\}$, we have $1/2\leq \lambda(s)\leq 3/2$, and using \eqref{ControlParam}, we deduce
$$
\lambda^{1/2} \left(1-\frac14\left(\frac{x_s}{\lambda}-1\right)\right)\geq \frac{1}{\sqrt{2}} \cdot  \frac34>\frac12.
$$
Moreover, it is easy to see that
$$
M_0=2\int \ep_0 Q +\int \ep_0^2\geq 2\int \ep_0 Q.
$$
Thus,
\begin{equation}\label{M_0term2}
\lambda^{1/2} \left(1-\frac14\left(\frac{x_s}{\lambda}-1\right)\right)M_0\geq  \int \ep_0Q.
\end{equation}
On the other hand, by Lemma \ref{Lemma-param} we have
$$
\left|\frac{\lambda_s}{\lambda}\right|+\left|\frac{x_s}{\lambda}-1\right|\leq C_2\|\ep(s)\|_2.
$$
Therefore, using the inequality \eqref{R-2}, there exists a universal constant $C_6>0$, such that
\begin{equation}\label{Rtilda2}
\lambda^{1/2} \widetilde{R}(\ep)\leq C_6\|\ep(s)\|^2_2.
\end{equation}
Now, by Lemma \ref{H^1-control}, we deduce
$$
\|\ep(s)\|^2_{H^1}\leq C_5\left(C_1\alpha_0 \left|\int \ep_0 Q\right|+\|\ep_0\|^2_{H^1}\right),
$$
and so the assumption \eqref{ep_0} yields
\begin{align}\label{ep(s)22}
\|\ep(s)\|^2_{H^1}\leq& C_5\left(C_1\alpha_0 \int \ep_0 Q+b_0\int \ep_0 Q \right)\nonumber \\
\leq & C_5(C_1+1) \left(\alpha_0+b_0\right)\int \ep_0 Q.
\end{align}
Therefore, collecting \eqref{Rtilda2}-\eqref{ep(s)22} and \eqref{Bound_ep}, there exists $C_8>0$ such that
\begin{align*}
\lambda^{1/2} \widetilde{R}(\ep)\leq & C_8\left(\alpha_0+b_0\right)\int \ep_0 Q.
\end{align*}

Now, let $\alpha_0, b_0>0$ sufficiently small such that  $C_8(\alpha_0+b_0)^{1/2}\leq 1/4$ and then in view of \eqref{K'M_0} and \eqref{M_0term2}, we deduce that
$$
\frac{d}{ds}K(s)\geq \frac12 \int \ep_0 Q>0 \quad \textrm{for all} ~~ s\geq 0,
$$
which gives the contradiction for large $s$ and finishes the proof.

\bibliographystyle{amsplain}

\begin{thebibliography}{99}


\bibitem{BSS}
J.L. Bona, P. Souganidis and W. Strauss,
\textit{Stability and instability of solitary waves of Korteweg-de Vries type},
Proc. Roy. Soc. London 411 (1987), 395--412.

\bibitem{CGNT}
S.-M. Chang, S. Gustafson, K. Nakanishi  and T.-P. Tsai,
\textit{Spectra of linearized operators for NLS solitary waves},
SIAM J. Math. Anal. 39 (2007/08), no. 4, 1070--1111.

\bibitem{Combet}
V. Combet, \textit{Construction and characterization of solutions converging to solitons for supercritical gKdV equations},
Diff. Integ. Eq. 23 (2010), no. 5-6, 513--568.

\bibitem{FHR2} L.G. Farah, J. Holmer and S. Roudenko,
\textit{Instability of solitons - revisited, II: the supercritical Zakharov-Kuznetsov equation}, preprint.


\bibitem{GSS} M. Grillakis, J. Shatah, and W. Strauss,
\textit{Stability theory of solitary waves in the presence of symmetry},
J. Funct. Anal. 74 (1987), 160--197.


\bibitem{KPV-93}
C. E. Kenig, G. Ponce and L. Vega,
\textit{Well-posedness and scattering results for the generalized Korteweg-de Vries equation via the contraction principle},
Comm. Pure Appl. Math. 46 (1993), 527--620.

\bibitem{K89}
M.K. Kwong,
Uniqueness of positive solutions of $\Delta u − u + u^p=0$ in $\cR^n$, Arch. Rational Mech. Anal. 105 (1989), no. 3, 243–266.

\bibitem{M02}
M. Maris,
Existence of nonstationary bubbles in higher dimensions,
J. Math. Pures Appl. (9) 81 (2002), no. 12, 1207–1239.

\bibitem{MM-KdV1}
Y. Martel and F. Merle,
\textit{A Liouville theorem for the critical generalized Korteweg–de Vries equation},
J. Math. Pures Appl. 79 (2000), 339--425.

\bibitem{MM-asym}
Y. Martel and F. Merle,
\textit{Asymptotic stability of solitons of the subcritical gKdV equations revisited},
Nonlinearity, 18 (2005), 55--80.

\bibitem{MM-KdV-instability}
Y. Martel and F. Merle,
\textit{Instability of solitons for the critical gKdV equation},
GAFA, Geom. Funct. Anal., 11 (2001), 74--123.

\bibitem{MM-KdV3}
Y. Martel and F. Merle,
\textit{Blow up in finite time and dynamics of blow up solutions for the $L^2$-critical generalized KdV equation},
J. Amer. Math. Soc. 15 (2002), 617--664.

\bibitem{MM-KdV4}
Y. Martel and F. Merle,
\textit{Stability of the blow up profile and lower bounds on the blow up rate for the critical generalized KdV equation},
Ann. Math. 155 (2002), 235--280.

\bibitem{M-KdV}
F. Merle, \textit{Existence of blow-up solutions in the energy space for the critical generalized Korteweg–de Vries equation},
J. Amer. Math. Soc. 14 (2001), 555--578.


\bibitem{W85}
M. Weinstein,
\textit{Modulational stability of ground states of nonlinear Schr\"odinger equations},
SIAM J. Math. Anal. 16  (1985),  no. 3, 472--491.



\end{thebibliography}

\end{document}